\documentclass[10pt,a4paper]{amsart}
\usepackage{amssymb}
\usepackage{hyperref}

\usepackage{graphicx,amssymb,amsfonts,epsfig,amsthm,a4,amsmath,url}
\usepackage[latin1]{inputenc}

\usepackage{amsmath, amsthm, amssymb, verbatim}
\usepackage{graphicx,amssymb,amsfonts,epsfig,amsthm,a4,amsmath,url,graphicx,enumerate}
\usepackage[latin1]{inputenc}
\usepackage{tikz}
\usepackage{textcomp}
\usetikzlibrary{matrix}

\newtheorem{thm}{Theorem}[section]
\newtheorem{cor}[thm]{Corollary}
\newtheorem{lem}[thm]{Lemma}

\newtheorem{prop}[thm]{Proposition}
\theoremstyle{definition}

\theoremstyle{remark}
\newtheorem{rem}[thm]{Remark}

\numberwithin{equation}{section}

\begin{document}
\title[Residually finite groups and box spaces]{From the geometry of box spaces to the geometry and measured couplings of groups}
\author[Das]{Kajal Das}
\address{UMPA UMR 5669, ENS Lyon\\  69364 Lyon cedex 7\\FRANCE}
\email{kajal.das@ens-lyon.fr}

\maketitle
\textbf{Abstract:} In this paper, we prove that if two `box spaces' of two residually finite groups are coarsely equivalent, then the two 
groups are `uniform measured equivalent' (UME). More generally, we prove that if there is a coarse embedding of one box space into another box 
space, then there exists a `uniform measured equivalent embedding' (UME-embedding) of the first group into the second one. This is a reinforcement of the 
easier fact that a coarse equivalence (resp.\ a coarse embedding)  between the box spaces gives rise to a coarse equivalence (resp.\ a coarse
embedding) between the groups.
 
We deduce new invariants that distinguish box spaces up to coarse embedding and coarse equivalence.
In particular,  we obtain that the expanders coming from $SL_n(\mathbb{Z})$ can not be coarsely embedded inside the expanders of 
$SL_m(\mathbb{Z})$, where $n>m$ and $n,m\geq 3$. Moreover, we obtain a countable class of residually groups which are mutually 
coarse-equivalent but any of their box spaces are not coarse-equivalent.

\textbf{Mathematics Subject Classification (2010):} 20E26, 20F69, 20F65, 37A15, 37A20, 51F99.

\textbf{Key terms:}: residually finite groups, box spaces, coarse equivalence, coarse embedding, measured equivalence, uniform measured equivalence (UME), 
uniform measured equivalent embedding (UME-embedding), Gromov-Hausdorff convergence,  marked groups.

\section{Introduction}

A finitely generated group is called \textit{residually finite} if there exists a decreasing sequence of finite index normal subgroups 
$\{G_n\}_{n\in\mathbb{N}}$ whose intersection is trivial, i.e.,  $G=G_1\trianglerighteq G_2\trianglerighteq\cdots\trianglerighteq G_n\trianglerighteq\cdots
\trianglerighteq{1}$. A \textit{box space} of a finitely generated residually finite group is defined as the disjoint union of 
$\sqcup_{n\in\mathbb{N}}G/G_n$. One can give a metric on the box space as follows: First, we fix a finite generating subset of $G$. 
This generating subset gives a metric on $G$ (which is the Cayley graph of $G$ with respect to the generating subset). 
Then, this metric induces a natural metric on each of $G/G_n$. We then assign a metric on the union such that the distance between two distinct copies 
$G/G_n$ and $G/G_{n+k}$  tends to infinity as $n\to \infty$. We refer to \cite{Kh} for a rigorous description of such a metric. 
Sometimes the box space of a group $G$ corresponding to a sequence of normal subgroups $\{G_n\}_{n\in\mathbb{N}}$ as above is denoted by $\square_{G_n}G$. 
The concept of box spaces plays an important role in the context of coarse Baum-Connes conjecture and maximal 
Baum-Connes conjecture (see \cite{Yu1}, \cite{CWY}). Also, the expander graphs obtained from Property (T) groups (see \cite{Mar}) form a significant class
of box spaces. 

 We recall that a map $f:X\rightarrow Y$ between two metric spaces is called $(\rho_{+}, \rho_{-})$-\textit{coarse embedding} if there exist two
 non-decreasing functions 
 $\rho_{+}$, $\rho_{-}:[0,\infty)\rightarrow [0,\infty)$ tending to $+\infty$ such that $\rho_{-}(d_X (x_1, x_2))\leq
d_Y (f (x_1), f (x_2))\leq \rho_{+}(d_X(x_1, x_2))$ for every $x_1,x_2\in X$; it is called a $(\rho_{+}, \rho_{-}, c)$-\textit{coarse equivalence} 
if moreover $Y$ is the $c$-neighbourhood of $f(X)$ for some $c\geq0$. The study of coarse embedding of a metric space into another metric space 
 mainly started from the the proof of `coarse Baum-Connes conjecture' by Yu in \cite{Yu1} for 
 every discrete metric space with `bounded geometry' which are coarsely embeddable inside a Hilbert space and this result implies 
 the Novikov conjecture for all closed manifolds whose fundamental group with the word-length metric with respect to some generating subset coarsely 
 embeds into a Hilbert space. We should mention that Yu's idea was implicit in Gromov's intution of approaching `Novikov's conjecture' in \cite{Gr1}.
 Yu also introduced `Property A' for metric spaces with bounded geometry and showed that the spaces with Property A coarsely embed into a Hilbert space. 
 This class of metric spaces includes Gromov's word hyperbolic groups, discrete subgroups of connected Lie groups and amenable groups.
 It was once conjectured that any `bounded geometry uniformly discrete' metric space can be coarsely embedded inside a Hilbert space. 
 Then, Gromov came up with a counter-example and proved that expanders can not be coarsely embedded inside a Hilbert space \cite{Gr3}, which
 follows immediately from an inequality of Matou\u{s}ek \cite{Mat}. 
 
There is a strong interaction between the analytic properties of residually finite groups and the coarse-geometric properties of the corresponding
box spaces. Let $G$ be a residually finite group and $\{G_n\}_{n\in\mathbb{N}}$ be a sequence of normal subgroups as above. 
Here is a general picture on the interactions between the residually finite groups and their box spaces.  

\bigskip
 
1. 
 \begin{center}
\begin{tikzpicture}[every node/.style={midway}]

\matrix[column sep={10em,between origins},
        row sep={2em}] at (0,0)
{ \node(A)   { $G$ is amenable}  ; & \node(B) {$\square_{G_i}G$ has Property $A$}; \\};

\draw[<->, double] (A) -- (B) node[anchor=south] {};

\end{tikzpicture}
\end{center}

2. \begin{center}
\begin{tikzpicture}[every node/.style={midway}]

\matrix[column sep={18em,between origins},
        row sep={2em}] at (0,0)
{ \node(A)   { $G$ has Haggerup}  ; & \node(B) {$\square_{G_i}G$ has a coarse fibred embedding into a Hilbert space}; \\
  \node(C) { }; & \node (D){$\square_{G_i}G$ has a coarse embedding into a Hilbert space}; \\};

\draw[<->, double] (A) -- (B) node[anchor=south] {};

\draw[<-, double] (B) -- (D) node[anchor=south] {};
\end{tikzpicture}
\end{center}
 
3. \begin{center}
\begin{tikzpicture}[every node/.style={midway}]

\matrix[column sep={15em,between origins},
        row sep={2em}] at (0,0)
{ \node(A)   { G has Property $\tau$ w.r.t. $\{G_i\}_{i\in\mathbb{N}}$}  ; & \node(B) {$\square_{G_i}G$ is an expander}; \\
  \node(C) {G has Property (T)}; & \node (D){$\square_{G_i}G$ has Geometric Property (T)}; \\};
\draw[->, double] (C) -- (A) node[anchor=east]  {};
\draw[<->, double] (A) -- (B) node[anchor=south] {};
\draw[<->, double] (C) -- (D) node[anchor=east]  {};
\draw[<-, double] (B) -- (D) node[anchor=south] {};
\end{tikzpicture}
\end{center}

We refer the readers to \cite{CWW},\cite{WY}, \cite{LZ}, \cite{Mar} for the above implications. However, Arzhantseva, Guentner and Spakula \cite{AGS} 
give the first example of a metric space 
(with bounded geometry) without Yu's Property A which coarsely embeds into a Hilbert space. This is constructed as a box space of $F_2$, the free group 
with two generators. Since any expander sequence can not be coarsely embedded inside a Hilbert space, Gromov asked the following natural question: 
Given a metric space with bounded geometry which does not embed coarsely into a  Hilbert space, does it necessarily contain a `weakly embedded' expander? 
Recently,  Arzhantseva and Tessera  in \cite{AT} answer this question negatively. This counter example also comes as a box space of a 
finitely generated residually finite group. We should remark that two box spaces of a group can have different coarse-geometric properties. 
For example, by Selberg's theorem (see \cite{LZ}) there exists a box space of $F_2$ which is an expander. On the other hand, using \cite{AGS}, we can construct another 
chain of normal subgroups of $F_2$ so that the corresponding box space coarsely embeds into a Hilbert space. Therefore, these two box spaces are not 
coarsely equivalent.

Our motivation in this article is to find obstructions for two box spaces to be coarsely equivalent (or to coarsely embed into one another).
It is not very difficult to prove that if there is a coarse equivalence (resp. coarse embedding) between two box spaces, then there is a
coarse equivalence (resp. coarse embedding) between the corresponding groups. We have proved this result in Section 2 (see \ref{coarseequi}). 
This result also appears in a general form for `marked groups' in \cite{KV}. We push this result one step further. We prove that
\textit{if two box spaces are coarsely equivalent, then the corresponding groups are `uniform measured equivalent'}(UME). More generally, we prove that if 
there is a coarse embedding of a box space into another box space, then there is a `UME-embedding' of the first group into the second one. 
\textit{Uniform Measured Equivalence} is a sub-equivalence relation of `Measure Equivalence' on finitely generated groups introduced by Shalom in \cite{Sh}.
Two countable discrete groups 
$\Gamma$ and $\Lambda$ are called \textit{Measured Equivalent}(ME) if they have commuting measure preserving free actions on a Borel space 
$(X,\mu)$ with finite measure Borel fundamental domains, say $X_\Gamma$ and $X_\Lambda$ respectively. The space $(X,\mu)$ is called a 
\textit{`measured coupling space'} for the groups $\Gamma$ and $\Lambda$. If, moreover, the action of an element of one group, say $\Gamma$, 
on the fundamental domain of another group, say $X_\Lambda$,  is covered by finitely many $\Lambda$-translates of $X_\Lambda$, then these two
groups are called UME. We refer the readers to Subsection \ref{2.1} for the definition of UME-embedding. 
Now, we formally state the main theorem of our paper.

\begin{thm}\label{mainthm} 
Suppose $G$ and $H$ are two finitely generated residually finite groups with two decreasing sequences of finite index normal subgroups 
${1}\trianglelefteq \cdots \trianglelefteq G_n\trianglelefteq\cdots\trianglelefteq G_1=G$ and 
${1}\trianglelefteq \cdots \trianglelefteq H_n\trianglelefteq\cdots\trianglelefteq H_1=H$. 
If the box space of $G$ with respect to $\{G_n\}_{n\in\mathbb{N}}$, denoted by $\square_{G_n} G$, and the box space of $H$ with respect to
$\{H_n\}_{n\in\mathbb{N}}$, denoted by $\square_{H_n} H$ are coarsely equivalent, then $G$ and $H$ are 
UME. Similarly, we prove that if there is a coarse embedding of 
$\square_{G_n} G$ into $\square_{H_n} H$, then there exists a `UME-embedding' of $G$ into $H$.
\end{thm}

As a consequence of the above theorem, we obtain the following invariants that distinguish box spaces up to coarse embedding and coarse equivalence.

\begin{cor}\label{affinecohdim}
 Suppose there is a coarse embedding of a box space of $G$ into a box space of $H$ as above. Then,
 \begin{itemize}
  \item[(i)] if there exists a metrically proper affine action of $H$ on an $L^p$-space, then there exists a metrically proper affine action of $G$ on an 
  $L^p$ space as well, where $1\leq p\leq\infty$;
  \item[(ii)] the $R$-cohomological dimension of $G$ is less than equal to the $R$-cohomological dimension of $H$, we denote it by $cd_R G \leq cd_R H$,
  where $R$ is a commutative ring containing $\mathbb{Q}$. 
 \end{itemize}
\end{cor}

By a deep result of Gaboriau (\cite{Ga}, Th\'{e}or\`{e}me 6.3), if two countable groups $\Gamma$ and $\Lambda$ are measured equivalent, 
then $\beta_i^{(2)}(\Gamma)= c \beta_i^{(2)}(\Lambda)$ for all $i\in\mathbb{N}$, where $c$ is a constant independent of $i$. Therefore, using Theorem 
\ref{mainthm}, we obtain the following corollary. 
\begin{cor}\label{l2betti}
   If two box spaces of $G$ and $H$ are coarsely equivalent, then the $l^2$-betti numbers of $G$ and $H$ must be proportional. 
\end{cor}

A special case of this corollary can be found in \cite{El} (see Proposition 2.1 and Proposition 3.4).

\subsection{Acknowledgements}
I am very grateful to my adviser Romain Tessera for suggesting me this problem (Theorem \ref{mainthm}) and helping me in the course of 
proving  the theorem. I am indebted to Ana Khukhro  and Alain Valette for inviting me at University of Neuchatel and having an illuminating discussion
on a similar project they were working on. I would also like to thank Damien Gaboriau for some very useful discussions and Damian Sawicki for pointing 
out some of the typos.

\subsection{Organization}
In Section 2, we introduce our necessary definitions, notations and abreviations. In Section 3, we prove that if there is a coarse equivalence 
between two box spaces, then the corresponding groups are coarsely equivalent. We prove
our main theorem \ref{mainthm} in Section 4. We divide this section into three subsections: In Subsection \ref{subsection4.1}, we construct a 
`topological coupling space' for the groups; in Subsection \ref{subsection4.2}, we give a non-zero measure on the topological coupling space
which is invariant under the actions of both groups, i.e., we make the topological coupling space into a measured coupling space; 
finally in Subsection \ref{subsection4.3} we prove our main theorem \ref{mainthm}. In Section 5, 
we discuss some applications of Proposition \ref{coarseequi} and Theorem \ref{mainthm}.

\section{Preliminaries: some definitions, notations and abreviations:}

\bigskip

\subsection{}\label{2.1} In \cite{Gr1}, Gromov first formulates a topological criterion for quasi-isometry and introduces measured equivalence as a measure theoretic
counterpart of quasi-isometry. In \cite{Sh}, Shalom slightly modifies the `topological coupling space' constructed by Gromov and also
gives a topological criterion for coarse embedding. We will mainly follow Shalom's construction of `topological coupling' in our proof. 
For countable groups $\Lambda$ and $\Gamma$, there exists a coarse embedding $\phi:\Lambda\rightarrow\Gamma$ if and only if there exists a locally compact 
space $X$ on which both $\Lambda$ and $\Gamma$ act properly and continuously with a compact-open fundamental domain $X_\Gamma$ of $\Gamma$ in $X$ 
and the actions of $\Gamma$ and $\Lambda$ commute. Replacing $\Gamma$ with a direct product $\Gamma\times M$ for some finite group $M$, we can assume 
that there exists a compact-open fundamental domain $X_\Gamma$ for $\Gamma$ satisfying $X_\Gamma\subseteq X_\Lambda$, where $X_\Lambda$ is a Borel
fundamental domain of $\Lambda$. Moreover,  $\phi$ will be a coarse 
equivalence between $\Gamma$ and $\Lambda$ if and only if after replacing $\Gamma$ with a direct product $\Gamma\times M$ for some finite group M, 
there exists a topological space $X$ with the following three properties: (i) both $\Lambda$ and $\Gamma$ act continuously, properly and freely on $X$;
(ii) there exist fundamental domains $X_\Gamma$ and $X_\Lambda$, for $\Gamma$ and $\Lambda$ respectively,  which are compact and open;
(iii) $X_\Gamma\subseteq X_\Lambda$ (\cite{Sh}, Theorem 2.1.2, p. 129). The space $X$ is called a \textit{topological coupling space} for $\Gamma$ and $\Lambda$.

In a similar way, as UME is defined in the previous section, we define \textit{uniform measured equivalence embedding} (UME-embedding) as follows:
We say that \textit{there is a UME-embedding of $\Lambda$ inside $\Gamma$} if there exists a Borel space $X$ with a non-zero mesure $\mu$ and measure 
preserving free commuting actions of $\Lambda$ and $\Gamma$ on $X$ such that there exists a finite measure fundamental domain $X_\Gamma$ of $\Gamma$ in $X$
and the action on $X_\Gamma$ by an element $\lambda$ from $\Lambda$ can be covered by finitely many $\Gamma$-translates of $X_\Gamma$
(the measure of $X_\Lambda$ may not necessarily be finite or the action on $X_\Lambda$ by an element $\gamma$ of $\Gamma$ may not necessarily be covered 
by finitely many $\lambda$-translates of $X_\Lambda$). 

Given an ME-coupling $(X,\mu)$ between $\Gamma$ and $\Lambda$, with $X_\Gamma$ and $X_\Lambda$ being the fundamental domains of $\Gamma$ and $\Lambda$ 
respectively, we define the cocyle $\alpha:\Gamma\times X_\Lambda\rightarrow \Lambda$  
(resp.\ $\beta:\Lambda\times X_\Gamma\rightarrow \Gamma$) by the rule: for all $x\in X_{\Lambda}$, and all 
$\gamma\in \Gamma$, $\alpha(\gamma, x)\gamma x \in X_\Lambda$ (and symmetrically for $\beta$).

\bigskip

\subsection{}\label{2.2} In this article, we denote the finite generating subset of $G$ by $S_G$, where $S_G$ is symmetric, i.e., $S_G=S_G^{-1}$.  
The metric defined by $S_G$ on $G$ will be denoted by $d_G$. We use $d_{G/G_n}$ for the induced metric on $G/G_n$.
We denote the identity elements of  $G$ and $G/G_n$ by $1_G$ and $1_{G/G_n}$, respectively. 
Sometimes, when the group is clear from the context, we use $1$ for the identity element of the group. 
The ball of radius $r$ around an element $g$ in $G$ will be  denoted by $B_r^G(g)$. We use similar type of notations for $H$ and $H/H_n$ as well. 
 
\bigskip

\subsection{}\label{2.3}  It is easy to observe that if we have a $(\rho_{+},\rho_{-},c)$-coarse equivalence (resp. $(\rho_{+}, \rho_{-})$-coarse embedding) 
of $\square_{G_n} G$ into $\square_{H_n} H$, 
then passing to a subsequence and reindexing the components, both for $\square_{G_n} G$ and $\square_{H_n} H$, 
we can assume that there is a $(\rho_{+},\rho_{-},c)$-coarse equivalence (resp. $(\rho_{+}, \rho_{-})$-coarse embedding) 
$f_n:=f|_{ G/G_n}$ from $G/G_n$ to  $H/H_n$ for all $n$.  Moreover, we can assume that 
$f(1_{G/G_n})=1_{H/H_n}$ for all $n$. 
We should remark that for finitely generated groups coarse equivalence coincides with quasi-isometry \cite{GK}. 

\bigskip  

\subsection{}\label{2.4} We now define Gromov-Hausdorff convergence (GH convergence) of compact metric spaces. We first need to define some terms before going into the 
definition of GH convergence. Suppose $X$ and $Y$ are two metric spaces and $f: X\rightarrow Y$ is a map. We define the `distortion' of $f$ by the 
following quantity: $$dis f := sup_{\{x_1,x_2\}} | d_Y(f(x_1),f(x_2))-d_X(x_1,x_2) |.$$ $f:X\rightarrow Y$ is called an $\epsilon$-isometry for 
some $\epsilon\geq0$ if $dis f \leq \epsilon$ and $Y$ is an $\epsilon$-neighborhood of $f(X)$. We sometimes say that `$f(X)$ is $\epsilon$-dense in 
$Y$' if $Y$ is an $\epsilon$-neighborhood of $f(X)$. There are several equivalent formulations of Gromov-Hausdorff 
convergence of compact metric spaces, we choose the following one for our purpose (see\cite{BBI}, p. 260): A sequence $\{X_n\}$ of compact metric spaces 
converges to a compact metric space $X$ if there is a sequence $\{\epsilon_n\}_{n\in\mathbb{N}}$ of positive numbers and a sequence of maps 
$f_n: X_n \rightarrow X$ such that every $f_n$ is an $\epsilon_n$-isometry and $\epsilon_n\rightarrow 0$ as $n\rightarrow\infty$.
 
\section{From coarse equivalence of two box spaces to the coarse equivalence of their corresponding groups}

\begin{prop}\label{coarseequi}
If $\square_{G_n} G$ and $\square_{H_n} H$ are coarsely equivalent, then $G$ and $H$ are also coarsely equivalent.
Similarly, if there is a coarse embedding of the box space  $\square_{G_n} G$ into the box space $\square_{H_n} H$, 
then there exists a coarse embedding of $G$ into $H$.
\end{prop}
\begin{proof} We give the proof of the first part of this proposition. The second part can be proved in a similar way.  
Let $\square_{G_n} G$ and $\square_{H_n} H$ be coarsely equivalent. From the discussion in Subsection \ref{2.3}, we can assume that 
there exists a $(\rho_{+}, \rho_{-}, c)$-coarse equivalence $\phi_n:G/G_n\rightarrow H/H_n$ with $\phi_n(1_{G/G_n})=1_{H/H_n}$ for all $n$.

For each ball $B_r^G(1)$ in $G$, we will construct a map $f_r$ from $B_r^G(1)$ into $B_{\rho_{+}(r)}^H(1)$ in $H$ so that they are compatible in the sense 
$f_r|_{B_{r-1}^G(1)}=f_{r-1}$ for all $r\in\mathbb{N}\setminus\{1\}$. The construction is as follows:

For each ball of radius $r$ in $G$, there exists an integer $m_r$ such that the projection from $G$ onto $G/G_{m_r}$ is injective 
on $B_r^G(1)$. Since $\phi_{m_r}$ is a $(\rho_{+}, \rho_{-}, c)$-coarse equivalence, the image $\phi_{m_r}[B_r^{G/G_{m_r}}(1)]$ is inside
$B_{\rho_{+}(r)}^{H/H_{m_r}}(1)$.
We also find a large integer $k_r$ such that the projection $H\twoheadrightarrow H/H_{k_r}$ is injective on $B_{\rho_{+}(r)}^H(1)$. 
Without loss of generality, we assume that $l_r:=m_r=k_r$. We define $\widetilde{\phi}_{l_r}: B_r^G(1)\rightarrow B_{\rho_{+}(r)}^H(1)$ 
so that the following diagram commutes: 

\begin{center}
\begin{tikzpicture}[node distance=2.5cm, auto]

  \node (A) {$B_r^G(1)(\subset G$)};
  \node (B) [below of=A] {$B_r^{G/G_{l_r}}(1)(\subset G/G_{l_r})$};
  \node (C) [node distance=4.5cm, right of=B] {$B_{\rho_{+}(r)}^{H/H_{l_r}}(1)(\subset H/H_{l_r})$};
  \node (D) [node distance=4.5cm, right of=A] {$B_{\rho_{+}(r)}^H(1)(\subset H)$};
  \draw[->] (A) to node {$\widetilde{\phi}_{l_r}$} (D);
  \draw[->] (A) to node {$\backsimeq$ (isometry)} (B);
  \draw[->] (B) to node {$\phi_{l_r}$} (C);
  \draw[->] (D) to node {$\backsimeq$ (isometry)} (C);
\end{tikzpicture}
\end{center}

Now, we consider the ball $B_1^G(1_G)$. If we restrict each of $\{\widetilde{\phi}_{l_r}\}_{r\in\mathbb{N}}$ on $B_1^G(1_G)$,
there exists a subsequence of $\{\widetilde{\phi}_{l_r}\}_{r\in\mathbb{N}}$ so that each of the functions in the subsequence coincides on 
$B_1^G(1_G)$. We take this subsequence and restrict each of 
the functions in this subsequence on $B_2^G(1_G)$. As before, we again extract another subsequence from $\{\widetilde{\phi}_{l_r}\}_{r\in\mathbb{N}}$ 
so that the functions in the subsequence agree on $B_2^G(1_G)$. In this way, using induction on the radius $r$ of the ball $B_r^G(1_G)$ and applying 
`Cantor's diagonal argument', we obtain $f_r: B_r^G(1)(\subseteq G) \rightarrow B_{\rho_{+}(r)}^H(1)(\subseteq H)$ such that $f_r|_{B_{r-1}^G(1)}=f_{r-1}$ 
for all $r\in\mathbb{N}\setminus\{1\}$.  Now, taking the limit of $f_r$ as $r\rightarrow\infty$, we construct a $(\rho_{+}, \rho_{-}, c)$-coarse equivalence
$\phi: G \rightarrow H$ with $\phi(1_G)=1_H$ such that there exists a subsequence $\{n_r\}_{r\in\mathbb{N}}$ satisfying the following diagram: 
\begin{center}
\begin{tikzpicture}[node distance=2.5cm, auto]

  \node (A) {$B_r^G(1)(\subset G$)};
  \node (B) [below of=A] {$B_r^{G/G_{n_r}}(1)(\subset G/G_{n_r})$};
  \node (C) [node distance=4.5cm, right of=B] {$B_{\rho_{+}(r)}^{H/H_{n_r}}(1)(\subset H/H_{n_r})$};
  \node (D) [node distance=4.5cm, right of=A] {$B_{\rho_{+}(r)}^H(1)(\subset H)$};
  \draw[->] (A) to node {$\phi$} (D);
  \draw[->] (A) to node {$\backsimeq$ (isometry)} (B);
  \draw[->] (B) to node {$\phi_{n_r}$} (C);
  \draw[->] (D) to node {$\backsimeq$ (isometry)} (C);
\end{tikzpicture}
\end{center}

 \end{proof}

 \begin{rem}\label{limitqi}
  In the abovementioned construction of $\phi$ from $\{\phi_n\}_{n\in\mathbb{N}}$, after passing through the subsequnce $\{n_r\}_{r\in\mathbb{N}}$ 
  and reindexing it, we obtain that
  
  \begin{center}
\begin{tikzpicture}[node distance=2.5cm, auto]

  \node (A) {$B_r^G(1)(\subset G$)};
  \node (B) [below of=A] {$B_r^{G/G_{n_r}}(1)(\subset G/G_{n_r})$};
  \node (C) [node distance=4.5cm, right of=B] {$B_{\rho_{+}(r)}^{H/H_{n_r}}(1)(\subset H/H_{n_r})$};
  \node (D) [node distance=4.5cm, right of=A] {$B_{\rho_{+}(r)}^H(1)(\subset H)$};
  \draw[->] (A) to node {$\phi$} (D);
  \draw[->] (A) to node {$\backsimeq$ (isometry)} (B);
  \draw[->] (B) to node {$\phi_{n_r}$} (C);
  \draw[->] (D) to node {$\backsimeq$ (isometry)} (C);
\end{tikzpicture}
\end{center}
We denote this construction of $\phi$ from $\{\phi_n\}_{n\in \mathbb{N}}$ by the notation 
$\{\phi_n\}_{n\in \mathbb{N}}\xrightarrow{\{n_r\}_{r\in \mathbb{N}}} \phi$.

 \end{rem}
 
\section{From coarse equivalence of two box spaces to UME of their corresponding groups}

\subsection{The construction of the topological coupling}\label{subsection4.1}

We suppose that there is a $(\rho_{+}, \rho_{-}, c)$-coarse equivalence $\phi$ between $\square_{G_n} G$ to $\square_{H_n} H$.
As discussed in \ref{2.3}, we assume that $\phi$ maps $G/G_n$ into $H/H_n$, the image $\phi[G/G_n]$ is $c$-dense in $H/H_n$ and 
$\phi(1_{G/G_n})=1_{H/H_n}$ for all $n$. 

\bigskip

We first recall the `topological coupling space' described in \cite{Ha} and \cite{Sh}. We define
$W=\{f:G\rightarrow H|f \hspace*{1mm}\mbox{is a} \hspace*{1mm}(\rho_{+},\rho_{-},c)-\mbox{coarse equivalence}\}$. 
By Proposition \ref{coarseequi}, $W$ is non-empty. We consider the set of all functions from $G$ to $H$ with pointwise convergence topology or with the 
topology of compact convergence.  We induce the subspace topology to $W$. Since the set of all functions from $G$ to $H$ with pointwise convergence topology
is a regular space with countable basis, by `Uryshon metrization theorem' it is metrizable \cite{Mu}. 
We, in particular, give the following metric on $W$: $$d(\phi,\psi)=2^{-sup\{r | \phi|_{B_r(1_G)}\equiv \psi|_{B_r(1_G)}\}}.$$
It is easy to see that the topology induced by this metric coincides with the pointwise convergence topology on $W$ and $W$ is a 
closed subset of the set of all functions from $G$ to $H$ with the pointwise convergence topology. Moreover, $W$ is a locally compact space \cite{Ha}.
We recall the argument briefly. Fix an element $\phi_0\in W$. We consider 
the neighborhood $\mathcal{U}_{(F,k)}:=\{\phi\in W |\hspace*{1mm} d_G\big(\phi(x),\phi_0(x)\big)\leq k \hspace*{1mm}\mbox{for all}\hspace*{1mm} x\in F\}$ of
$\phi_0$ corresponding to a non-empty finite subset $F$ in $G$ and a number $k>0$.  Now, by Ascoli's theorem, $\mathcal{U}_{(F,k)}$ is a 
compact neighborhood of $\phi_0$. 

\bigskip

We define $G$ and $H$ (left) actions  on a function $\phi$ in $W$ as follows:
$$gh\phi(x)= h\phi(g^{-1}x)$$
Let $Y:=\{\phi:G\rightarrow H | \phi\in W \hspace*{1mm}\mbox{and}\hspace*{1mm}\phi(1_G)=1_H\}$. 
Again using Proposition \ref{coarseequi}, we know that $Y$ is non-empty.
We induce the metric of $W$ on $Y$. We define $G$ action on $Y$ by
$$[g\cdot \phi](g')=  [\phi(g^{-1})]^{-1} \phi(g^{-1}g').$$ 
It is easy to see that $G$ actions on $W$ and $Y$ are continuous.  
By \cite{Ha} (p. 99), we know that $W$ is a `topological coupling space' of $G$ and $H$, where $Y$ is a compact fundamental domain of $H$. 

\bigskip

 If there is a  $(\rho_{+}, \rho_{-}, c)$-coarse equivalence (resp. $(\rho_{+}, \rho_{-})$-coarse embedding) $\phi_n:G/G_n \rightarrow H/H_n$ for all $n$,
 then there exists a finite group $M$, independent of $n$, such that there is an injective coarse equivalence (resp. coarse embedding) 
 $\widetilde{\phi_n}:G/G_n \rightarrow H/H_n\times M$ for all $n$ with some modification of $(\rho_{+}, \rho_{-}, c)$ (resp. $(\rho_{+}, \rho_{-})$), 
 which we denote by the same notation as before. We define 
 $Z_n:=\{f: G/G_n\rightarrow H/H_n\times M | f\hspace*{1mm}\mbox{is injective and}\hspace*{1mm} f \hspace*{1mm}\mbox{is a} $
 $(\rho_{+},\rho_{-},c)-\mbox{coarse equivalence}\}$
 and $X_n:=\{f\in Z_n | f(1_{G/G_n})=1_{H/H_n}\}$. In a similar way, as defined on $W$, we give a metric $d_n$ on each $Z_n$ and a $G$-action on it 
 (through the action of $G/G_n$). By the construction given in \ref{coarseequi}, for a sequence of injective $(\rho_{+},\rho_{-},c)$-coarse equivalence 
 (resp. injective $(\rho_{+},\rho_{-})$-coarse embedding) $f_n:G/G_n\rightarrow H/H_n\times M$ for all $n$, we will have an injective 
 $(\rho_{+},\rho_{-},c)$-coarse equivalence (resp. injective $(\rho_{+},\rho_{-})$-coarse embedding) $\widetilde{f}$ from $G$ into 
 $H\times M$. For the notational convenience, in the remaining part of this subsection and in 
 Subsection \ref{subsection4.2}, we will replace $H_n\times M$ by $H_n$ and $H\times M$ by $H$.

 \bigskip 
 
For our purpose, we will construct another topological coupling $Z(\subseteq W)$
for $G$ and $H$. We will take the Gromov-Haudorff limit of the collection of finite metric spaces $X_n$ 
(possibly after passing to a subsequence) and identify the limit with a $G$-invariant compact subset $X$ of $Y$ (Proposition \ref{ghconv}), 
which will turn out to be a new fundamental domain for $H$, and we will define the new `topological coupling space' as $Z= H\hspace*{1mm}X$. 

\bigskip

\begin{prop}\label{ghconv}
 There exists a subsequence of $\{X_n\}_{n\in\mathbb{N}}$ which converges to a $G$-invariant compact subspace $X$ of $Y$ in Gromov-Hausdorff topology. 
\end{prop}

We refer the reader to  \ref{2.4}  for the definition of Gromov-Hausdorff topology. 
Before going into the proof of \ref{ghconv}, we prove the following lemmas. 

\begin{lem}\label{relcomp}
 The sequence $\{X_n\}_{n\in\mathbb{N}}$ of compact metric spaces is `uniformly totally bounded', i.e.,
 
\begin{itemize}
 \item  there is a constant $D$ such that $\mbox{diam}\hspace*{1mm} X_n \leq D$ for all $n$;
 \item  for all $\epsilon>0$ there exists a natural number $N=N(\epsilon)$ such that
every $X_n$ contains an $\epsilon$-net consisting of at most $N$ points.
\end{itemize}
Sometimes, this type of sequence $\{X_n\}_{n\in\mathbb{N}}$ is also called `relatively compact' in Gromov-Hausdorff metric. 
\end{lem}
\begin{proof}
Fix $\epsilon >0$. We choose an arbitrarily large integer $R$ such that $2^{-R}<\epsilon$.
Since $\mbox{diam}\hspace*{1mm}  X_n \leq 1$ for all $n$, the first criterion is trivially satisfied. We now prove the second criterion.
We construct a map 
$$h_n:X_n\rightarrow \mathcal{F}\Big(B^{G/G_n}_R(1_{G/G_n}), B^{H/H_n}_{\rho_{+}(R)}(1_{H/H_n})\Big)\hspace*{3mm} \mbox{ defined by}\hspace*{3mm} \xi \mapsto \xi|_{B^{G/G_n}_R(1_{G/G_n})},$$
where $\mathcal{F}\Big(B^{G/G_n}_R(1_{G/G_n}), B^{H/H_n}_{\rho_{+}(R)}(1_{H/H_n})\Big)$ is the collection of all maps from 
\linebreak
the ball of radius $R$ in $G/G_n$ around $1_{G/G_n}$, denoted by $B^{G/G_n}_R(1_{G/G_n})$,  to the ball of radius $\rho_{+}(R)$ in $H/H_n$ around 
$1_{H/H_n}$, denoted by $B^{H/H_n}_{\rho_{+}(R)}(1_{H/H_n})$. The map is well-defined because the elements of $X_n$ are $(\rho_{+}, \rho_{-}, c)$-coarse equivalence.
For all $X_n$, we can  choose a subset $A_n$ of $X_n$ such that the restriction of $h_n$ on $A_n$ is injective and the image of $A_n$ under $h_n$ covers the
whole image of $X_n$.

\bigskip

We claim that $A_n$ is an $\epsilon$-net in $X_n$.  Take any $\xi$ in $X_n$. By the definition of $A_n$, there exists an element 
$\eta\in A_n$ so that $\xi$ and $\eta$ belong to the same fiber of the map $h_n$. Therefore, $\xi$ and $\eta$ coincides on the ball $B^{G/G_n}_R(1_G)$,
which implies that $d_X(\xi,\eta)\leq 2^{-R}<\epsilon$. Hence, $A_n$ is an $\epsilon$-net in $X_n$.

\bigskip

Now, we compute an estimate of the cardinality of the set $A_n$. We observe that 
\begin{eqnarray*}
\Big|\mathcal{F}\Big(B^{G/G_n}_R(1_{G/G_n}), B^{H/H_n}_{\rho_{+}(R)}(1_{H/H_n})\Big)\Big| & \leq & | B^{G/G_n}_R(1_{G/G_n})|\hspace*{1mm}|B^{H/H_n}_{\rho_{+}(R)}(1_{H/H_n})|\\
                                                         & \leq & |S_G|^R\hspace*{1mm}|S_H|^{\rho_{+}(R)},
\end{eqnarray*}
where $S_G$ and $S_H$ are two generating subsets of $G$ and $H$ respectively. Therefore, $|A_n|\leq |S_G|^R\hspace*{1mm}|S_H|^{\rho_{+}(R)}$ for all $n$.
Hence, $\{X_n\}_{n\in\mathbb{N}}$  is uniformly totally bounded.
\end{proof}

\begin{lem}\label{gpactionconv}
  $$\{\phi_n\}_{n\in \mathbb{N}}\xrightarrow{\{l_r\}_{r\in \mathbb{N}}} \phi \hspace*{1mm}\Rightarrow\hspace*{1mm}\{g\cdot \phi_n\}_{n\in\mathbb{N}}
  \xrightarrow{\{l_{m_r}\}_{r\in \mathbb{N}}} g\cdot \phi,$$ 
 where $\phi_n\in X_n$, $\phi\in X$ and $m_r=max\{r+|g^{-1}|, \left \lceil{\rho_{+}(r+|g^{-1}|)}\right \rceil\}$. 
 For the notation $\{\phi_n\}_{n\in \mathbb{N}}\xrightarrow{\{l_r\}_{r\in \mathbb{N}}} \phi$, we refer the reader to Remark \ref{limitqi}.
\end{lem}
\begin{proof} First, we recall the definition of $g\cdot\phi_n$ and $g\cdot \phi$, where $\phi_n\in X_n$ and $\phi\in X$. They are defined as follows: 
$$[g\cdot\phi_n](\bar{x})=[\phi_n(g^{-1})]^{-1}\phi_n(g^{-1}\bar{x})\hspace*{1mm}\mbox{and}\hspace*{1mm}[g\cdot\phi](x)=[\phi(g^{-1})]^{-1}\phi(g^{-1} x),$$
where $\bar{x}\in G/G_n$ and $x\in G$. Now, the lemma easily follows from the following commuting diagram: 
\begin{center}
\begin{tikzpicture}[node distance=2.5cm, auto]

  \node (A) {$g^{-1}B_r^G(1)(\subset G$)};
  \node (B) [below of=A] {$g^{-1}B_r^{G/G_{l_{m_r}}}(1)(\subset G/G_{l_{m_r}})$};
  \node (C) [node distance=6 cm, right of=B] {$B_{\rho_{+}(r+|g^{-1}|)}^{H/H_{l_{m_r}}}(1)(\subset H/H_{l_{m_r}})$};
  \node (D) [node distance=6 cm, right of=A] {$B_{\rho_{+}(r+|g^{-1}|)}^H(1)(\subset H)$};
  \draw[->] (A) to node {$\phi$} (D);
  \draw[->] (A) to node {$\backsimeq$ (isometry)} (B);
  \draw[->] (B) to node {$\phi_{l_{m_r}}$} (C);
  \draw[->] (D) to node {$\backsimeq$ (isometry)} (C);
\end{tikzpicture}
\end{center}
\end{proof}

\textit{Proof of Proposition \ref{ghconv}:}\label{proofghconv}
In the proof of this proposition, we will be using some ideas of the construction of a limiting compact metric space for a `uniformly totally bounded' sequence 
of compact metric spaces (see Theorem 7.4.15, p. 274, \cite{BBI}).

\bigskip

\textit{Step 1:} By Lemma \ref{relcomp}, there exists a countable dense collection $S_n = \{x_{i,n}\}_{i=1}^\infty$ in each $X_n$ such that 
for every $k$ the first $N_k$ points of $S_n$, denoted by $S_n^{(k)}$, form a $(1/k)$-net in $X_n$. Without loss of generality, we assume that 
$S_n^{(k)}\subset S_n^{(k+1)}$ for all $k\in\mathbb{N}$. Using Proposition \ref{coarseequi} and Cantor's diagonal argument, after passing through a 
subsequence, we obtain that 
$$\{x_{i,n}\}_{n\in \mathbb{N}}\xrightarrow{\{n\}_{n\in \mathbb{N}}} x_i,$$
for some $x_i\in X$ and for all $i\in\mathbb{N}$.
Let $S^{(k)}:=\{x_i|i=1,\dots,N_k\}\subseteq X$ for all $k$ and $S:=\cup_{k=1}^\infty S^{(k)}$. We define 
$X=\overline{\{x_i\}_{i\in\mathbb{N}}}\subseteq Y$ and $X':=\overline{G\cdot \{x_i\}_{i\in\mathbb{N}}}\subseteq Y$. 
Since $X$ and $X'$ are closed subsets of the compact set $Y$, both $X$ and $X'$ are compact, and by definition
$X'$ is $G$-invariant.

\bigskip

\textit{Step 2:} In the rest of the paper, we denote the distance functions in $X_n$ and $X$ by $d_{X_n}$ and $d_X$, respectively. 
The distance $d_{X_n}(g\cdot x_{i,n}, g'\cdot x_{j,n})$ does not exceed $1$, i.e., belongs to a compact interval. Therefore, using 
`Cantor's diagonal procedure', we extract a subsequence of $\{X_n\}_{n=1}^{\infty}$ such that after passing through the subsequence 
$\{d_{X_n}(g\cdot x_{i,n},g'\cdot x_{j,n})\}_{n=1}^{\infty}$ converges for all $i, j\in\mathbb{N}$ and for all $g, g'\in G$. Moreover, using Lemma 
\ref{gpactionconv}, after passing through another subsequence, say $\{n_m\}_{m\in\mathbb{N}}$, we obtain that 
$$\{g\cdot x_{i,n}\}_{n\in \mathbb{N}}\xrightarrow{\{n_m\}_{m\in \mathbb{N}}} g\cdot x_i\hspace*{1mm}\mbox{and}\hspace*{1mm}\lim_{m\rightarrow\infty} 
d_{X_{n_m}}(g\cdot x_{i,n_m}, g'\cdot  x_{j,n_m})=d_X(g\cdot x_i, g'\cdot x_j),$$
for all $i, j\in\mathbb{N}$ and for all $g, g'\in G$.

\bigskip

\textit{Step 3:} We claim that $S^{(2k)}$ is a $(1/k)$-net in $X$ and $X'$. Indeed, every set $S_n^{(k)}= \{x_{i,n} | i=1,\dots, N_k\}$ is a $(1/k)$-net
in the respective space $X_n$ . Hence, for every $g\cdot x_{i,n}\in  X_n$  there is a $j\leq N_k$ such
that $d_{X_n}(g\cdot x_{i,n}, x_{j,n})\leq 1/k$. Since $N_k$ does not depend on $n$, for every fixed $g\in G$ and $i\in\mathbb{N}$,
there is a $j\leq N_k$ such that $d_{X_n}(g\cdot x_{i,n}, x_{j,n})\leq 1/k$ for infinitely many indices $n$.
Passing to the limit, we obtain that $d_X(g\cdot x_i, x_j)\leq 1/k$ for this $j$. Thus, $S^{(2k)}$ is a
$(1/k)$-net in $X$ and $X'$ for all $k$. Moreover, we obtain that $X=X'$, which implies that $X$ is $G$-invariant.  

\bigskip

\textit{Step 4:} Since, by Step 2, $d_X(x_{i,n_m},x_{j,n_m})\xrightarrow{m\rightarrow\infty}d_X(x_i,x_j)$ for all $i,j$, we obtain that $S_{n_m}^{(k)}$
converges to $S^{(k)}$ in GH-topology as $m\rightarrow\infty$ for all $k\in\mathbb{N}$. Now, since $S_{n_m}$ is dense in $X_{n_m}$
and $S$ is dense in $X$, we have $X_{n_m}$ converges to $X$ in GH-topology. Hence, we have our proposition. 
\hfill\(\Box\)

\subsection{Construction of a $G$-invariant measure on $X$}\label{subsection4.2}

In this section, we construct a $G$-invariant probability measure $\mu$ on $X$.
%We give a $G$-invariant probability measure $\mu_n$ on each $X_n$ and take  a `limit of these measures' on $X$ in a suitable sense. 
The idea of the construction of this measure uses the concept of Gromov's 
$\underline\Box_{1}$-convergence for metric-measure spaces (\cite{Gr2} p. 118).

\begin{prop}\label{measureconv}
Let $\{(X_n,d_{X_n})\}_{n\in\mathbb{N}}$ be the sequence of compact metric spaces which converges to $(X,d_X)$ in Gromov-Hausdorff topology as obtained by
Proposition \ref{ghconv}. 
We give a $G$-invariant probability measure $\mu_n$ on each $X_n$. Then, there exists a $G$-invariant probability measure $\mu$ on $X$.
\end{prop}

We prove Proposition \hspace{0.5mm}\ref{measureconv}\hspace{0.5mm} at the end of this subsection. 
We prove this proposition by taking the weak* limit of the pushward measure of $\mu_n$ on $X$ by a sequence of suitable `$\xi_n$-isometry' from 
$X_n$ to $X$ (possibly after passing through a subsequence), where $\xi_n\rightarrow 0$ as $n\rightarrow\infty$. We need `almost $G$-equivariant' 
$\xi_n$-isometry from $X_n$ to $X$ to obtain $G$-invariance of the limiting measure. 
We prove the existence of such `almost $G$-equivariant' $\xi_n$-isometries in the following proposition.

\begin{prop}\label{gequifn}
There exist a subsequence $\{n_k\}_{k\in\mathbb{N}}$ and $\xi_k$-isometries $f_k:X_{n_k}\rightarrow X$ for all $k\in\mathbb{N}$
such that $sup_{x\in X_{n_k}}\hspace*{0.5mm}d_{X_{n_k}}\big(g\cdot f_k(x), f_k(g\cdot x)\big)< \xi_k$
for all $g\in B_k^G(1_G)$, where $\xi_k\rightarrow 0$ as $k\rightarrow\infty$. 
\end{prop}

Before going into the proof of Proposition \ref{gequifn}, we prove the following lemmas. 

\begin{lem}\label{equicont}
The actions of $G$ on $\{X_n | n\in\mathbb{N} \}\cup\{X\}$ are equicontinuous, i.e., $g\in G$ being fixed, for all $\epsilon>0$ there exists a $\delta>0$ 
such that if $d_{X_n}(\phi_n,\psi_n)< \epsilon$ (resp. $d_X(\phi,\psi)<\epsilon$), then $d_{X_n}\big(g\cdot\phi_n,g\cdot\psi_n\big)< \delta$
(resp. $d_X\big(g\cdot\phi,g\cdot\psi\big)< \delta$), where $\delta$ only depends on $g$ and $\epsilon$, and 
$\delta\rightarrow 0$ as $\epsilon\rightarrow 0$. 
\end{lem}
\begin{proof}
 We fix $g\in G$ and $\epsilon >0$. Without loss of generality, we assume that $\epsilon< 2^{-|g|}$. If $d_{X_n}(\phi_n,\psi_n)=2^{-R}<\epsilon$,
 then $d_{X_n}\big(g\cdot \phi_n, g\cdot\psi_n\big)\leq 2^{-(R-|g|)}$. 
 Therefore, the lemma follows for $\{X_n | n\in\mathbb{N} \}$ by taking $\delta=\epsilon\hspace*{1mm} 2^{|g|}$. The statement for $X$ follows by 
 passing to the limit. 

\end{proof}

\begin{lem}\label{gequiviso}
 Let $k\in\mathbb{N}$ and $r\in\mathbb{N}$ be two fixed numbers. Then, there exists a subsequence $\{n_m\}_{m\in\mathbb{N}}$ such that
 $B_r^G(1_G)\cdot S_{n_m}^{(k)}$ converges to $B_r^G(1_G)\cdot S^{(k)}$ 
 in GH-topology and there exists an $\epsilon_m$-isometry $f_m^{(r,k)}:B_r^G(1_G)\cdot S_{n_m}^{(k)}\rightarrow B_r^G(1_G)\cdot S^{(k)}$ for all $m$ such that 
 $f_m^{(r,k)}|_{S_{n_m}^{(k)}}$ is $B_r^G(1_G)$-equivariant, i.e., $f_m^{(r,k)}(g\cdot x_{i,n_m})= g \cdot f_m^{(r,k)}(x_{i,n_m})$ for all $g\in B_r^G(1_G)$ and for all $i\in\{1,\ldots, N_k\}$, and
 $\epsilon_m\rightarrow 0$ as $m\rightarrow\infty$.   
\end{lem}
\begin{proof} Using the same argument as given in Step 2 of the proof of Proposition \ref{ghconv}, after passing through a
subsequence, we have 
$$\{g\cdot x_{i,n}\}_{n\in \mathbb{N}}\xrightarrow{\{n\}_{n\in \mathbb{N}}} g\cdot x_i\hspace*{1mm}\mbox{and}\hspace*{1mm}\lim_{n\rightarrow\infty} 
d_{X_n}(g\cdot x_{i,n}, g'\cdot  x_{j,n})=d_X(g\cdot x_i, g'\cdot x_j),$$
for all $i,j\in\{1,\ldots, N_k\}$ and for all $g, g'\in B_r^G(1_G)$. 
We fix $g\in G$ and $i\in \{1,\ldots, N_k\}$. If $g\cdot x_{i,n}\in S_n^{(k)}$ for 
infinitely many $n\in\mathbb{N}$, we find a subsequence 
$\{n_u\}_{u\in\mathbb{N}}$ and $x_{j, n_u}\in S_{n_u}^{(k)}$ such that $g\cdot x_{i,n_u}=x_{j,n_u}$ for some $j\in \{1,\ldots, N_k\}$
and for all $u\in\mathbb{N}$. Therefore, we have 
$$\{g\cdot x_{i,n_u}=x_{j,n_u}\}_{u\in \mathbb{N}}\xrightarrow{\{n_u\}_{u\in \mathbb{N}}} g\cdot x_i=x_j.$$
If this is not the case, we obtain another subsequence $\{n_v\}_{v\in\mathbb{N}}$ such that 
$g\cdot x_{i, n_v}\notin S_{n_v}^{(k)}$ for all $v\in\mathbb{N}$. By Lemma \ref{gpactionconv}, we obtain a subsequence
$\{n_w\}_{w\in\mathbb{N}}$ of $\{n_v\}_{v\in\mathbb{N}}$ such that 
$$\{g\cdot x_{i,n_w}\}_{w\in \mathbb{N}}\xrightarrow{\{n_w\}_{w\in \mathbb{N}}} g\cdot x_i.$$ 
Applying the above procedure inductively on the elements of $g\in B_r^G(1_G)$ and $i\in\{1,\ldots, N_k\}$, we obtain a subsequence $\{n_m\}_{m\in\mathbb{N}}$ 
such that 
$$\{g\cdot x_{i,n}\}_{n\in \mathbb{N}}\xrightarrow{\{n_m\}_{m\in \mathbb{N}}} g\cdot x_i\hspace*{1mm}\mbox{and}
\hspace*{1mm}\lim_{m\rightarrow\infty}d_{X_{n_m}}(g\cdot x_{i,n_m}, g'\cdot  x_{j,n_m})=d_X(g\cdot x_i, g'\cdot x_j),$$
for all $g, g'\in B_r^G(1_G)$ and for all $i, j\in \{1,\ldots, N_k\}$. We define  
$f_m^{(r,k)}: B_r^G(1_G)\cdot S_{n_m}^{(k)}\rightarrow B_r^G(1_G)\cdot S^{(k)}$
by mapping $g\cdot x_{i,n_m}\mapsto g\cdot x_i$. This is crucial to observe that $f_m^{(r,k)}$ is a well-defined map. Now, since 
$d_{X_{n_m}}(g\cdot x_{i,n_m}, g'\cdot  x_{j,n_m})\xrightarrow{m\rightarrow\infty}d_X(g\cdot x_i, g'\cdot x_j)$, therefore $f_m^{(r,k)}$ is an $\epsilon_m$-isometry
for some $\epsilon_m>0$, where $\epsilon_m\rightarrow 0$ as $m\rightarrow\infty$. 
\end{proof}

\bigskip

\textit{Proof of Proposition \ref{gequifn}:}
Using Lemma \ref{gequiviso} and Cantor's diagonal procedure, we obtain a subsequence $\{n_k\}_{k\in\mathbb{N}}$ and $\epsilon_k$-isometry 
$f_k: B_k^G(1_G)\cdot S_{n_k}^{(k)}(\subset X_{n_k}) \rightarrow B_k^G(1_G)\cdot S^{(k)}(\subset X)$ such that $\epsilon_k\rightarrow 0$ as 
$k\rightarrow\infty$. Without loss of generality, we assume that $\epsilon_k>1/k$ for all $k\in\mathbb{N}$. 
Fix $k\in\mathbb{N}$ and $g\in B_k(1_G)$. We extend $f_k$ on $X_{n_k}$ in the following way (we denote the extended map by the same symbol $f_k$): 
Let $x\in X_{n_k}\setminus B_k^G(1_G)\cdot S_{n_k}^{(k)}$. Then, there exists $x_{i,n_k}\in S_n^{(k)}$ such that  $d_{X_{n_k}}(x, x_{i,n_k})<1/k$. 
We define $f_k(x):=f_k(x_{i,n_k})$. It is easy
to observe that the extended map $f_k$ is a $3\hspace*{0.5mm}\epsilon_k$-isometry.
Let $\delta_k>0$ be the number corresponding to $3\hspace*{0.5mm}\epsilon_k$ obtained by Lemma \ref{equicont}. Therefore, 
$d_{X_{n_k}}(g\cdot x, g\cdot x_{i,n_k})<\delta_k$ and $d_X \big(g\cdot f_k(x), g\cdot f_k(x_{i,n_k})<\delta_k$. 
Since $f_k$ is a $3\hspace*{0.5mm}\epsilon_k$-isometry, $d_X(f_k(g\cdot x),  f_k(g\cdot x_{i,n_k})) < 3\hspace*{0.5mm}\epsilon_k+\delta_k$. 
Now, we have the following inequality:
\begin{eqnarray*}
d_X\big(g\cdot f_k(x), f_k(g\cdot x)\big) & \leq & d_X \big(g\cdot f_k(x), g\cdot f_k(x_{i,n_k})\big)\\
                                                      &  & + d_X\big(g\cdot f_k(x_{i,n_k}), f_k(g\cdot x_{i,n_k})\big)\\
                                                      &  & + d_X\big(f_k(g\cdot x_{i,n_k}), f_k(g\cdot x)\big).
\end{eqnarray*}
By Lemma \ref{gequiviso}, $d_X\big(g\cdot f_k(x_{i,n_k}), f_k(g\cdot x_{i,n_k})\big)=0$. Therefore, we obtain that 
\begin{equation}\label{g-equi-fn}
 d_X\big(g\cdot f_k(x), f_k(g\cdot x)\big)< 3\hspace*{0.5mm}\epsilon_k+2 \delta_k.
\end{equation}
Hence, we have our proposition by taking $\xi_k=3\hspace*{0.5mm}\epsilon_k+ 2 \delta_k$. \hfill\(\Box\)

\begin{cor}\label{gequiinv}
Using the notations of Proposition \ref{gequifn}, we have
$$d_{X_{n_k}}^{Haus}\big(g\cdot f_k^{-1}(A) ,f_k^{-1}( g\cdot A)\big)\leq 2\hspace*{0.5mm}\xi_k$$ 
for all $g\in B_k^G(1_G)$ and for all subsets $A$ of $X$, where $d_{X_{n_k}}^{Haus}$ denotes the Hausdorff distance between two subsets 
in $X_{n_k}$. 
\end{cor}
\begin{proof}

Fix $y\in A$. Let $z=g\cdot x\in g\cdot f_k^{-1}(\{y\})$ and $z'\in f_k^{-1} (\{g\cdot y\})$. Since $f_k$ is a $\xi_k$-isometry, we have
$$d_{X_{n_k}}(z,z')\leq \xi_k + d_X\big(f_k(z), f_k(z')\big).$$
We observe that $f_k(z)=f_k(g\cdot x)$ and $f_k(z')=g\cdot f_k(x)$. Now, using equation \ref{g-equi-fn}, we obtain that 
$d_{X_{n_k}}(z,z')\leq 2\hspace*{0.5mm}\xi_k$. Hence, we have our corollary. 
\end{proof}

\bigskip

\textit{Proof of Proposition \ref{measureconv}:}

By Proposition \ref{gequifn}, we obtain a subsequence $\{n_k\}_{k\in\mathbb{N}}$ and $\xi_k$-isometries $f_k:X_{n_k}\rightarrow X$ for all $k\in\mathbb{N}$
such that $sup_{x\in X_{n_k}}\hspace*{0.5mm}d_{X_{n_k}}\big(g\cdot f_k(x), f_k(g\cdot x)\big)\leq \xi_k$
for all $g\in B_k^G(1_G)$, where $\xi_k\rightarrow 0$ as $k\rightarrow\infty$. For our convenience, we take $\mu_n$ as the uniform measure on $X_n$.
For all $k\in\mathbb{N}$, we define $\widetilde{\mu_k}:=f_k^{*}(\mu_{n_k})$,
the pushforward measure of $\mu_{n_k}$ on $X$ by $f_k$ . We consider the space of all probability measures on $X$ with weak* topology, which we denote by 
$\mathcal{P}(X)$. By Banach-Alaoglu theorem, $\mathcal{P}(X)$ is compact 
in weak* topology. Therefore, there exists a subsequence of $\{\widetilde{\mu_k}\}_{k=1}^\infty$ which converges to a probability measure, say $\mu$, on $X$.
Without loss of generality, we denote the subsequence by the same notation $\{\widetilde{\mu_k}\}_{k=1}^\infty$. We will prove that $\mu$ is $G$-invariant. 

It is a known fact from (\cite{Gr2}, \cite{DV} p. 398) that the weak* topology on $\mathcal{P}(X)$ is metrizable and the metric is given by the following 
Prokhorov metric: $d_{P}^{X}(\lambda,\nu):= inf\{\eta>0 | \lambda(A)\leq \nu(A^{\eta})+\eta \hspace*{2mm}\mbox{and} \hspace*{2mm}\nu(A)\leq
\lambda(A^{\eta})+\eta\}$, where $A^\eta$ is the $\eta$-neighborhood of $A$. Since the $\sigma$-algebra of $X$ is generated by the countable number of 
clopen subsets of $X$, it suffices to prove $g\cdot \mu(A)=\mu(A)$ for all clopen subsets $A$ of $X$. We fix a clopen subset $A$ of $X$ and $g\in G$. 
There exists $k_0\in\mathbb{N}$ such that $g\in B_k^G(1_G)$ for all $k\geq k_0$. Since $A$ is a clopen set, $A^\eta=A$ for sufficiently small $\eta$.
Now, from the definition of $g\cdot\mu$ and $\mu$, we get
\begin{equation}
(g\cdot \mu)(A)=\mu (g^{-1}\cdot A)= \lim_{k\rightarrow\infty}\widetilde{\mu_k}\big(g^{-1}\cdot A\big)
=\lim_{k\rightarrow\infty}\mu_{n_k}\big(f_k^{-1}([g^{-1}\cdot A])\big). 
\end{equation}

Now, using Corollary \ref{gequiinv}, we have
\begin{equation}
 f_k^{-1}( g^{-1}\cdot A)\subseteq [g^{-1}\cdot f_k^{-1}(A)]^{2\hspace*{0.5mm}\xi_k}
\end{equation}
By Lemma \ref{equicont}, corresponding to each number $2\hspace*{0.5mm}\xi_k$, we obtain a positive number $\delta_k$ tending to zero as 
$k\rightarrow\infty$ such that
\begin{equation}
 g\cdot [g^{-1}\cdot f_k^{-1}(A)]^{2\hspace*{0.5mm}\xi_k}\subseteq [f_k^{-1}(A)]^{\delta_k}
\end{equation}
Since $\mu_{n_k}$ is $G$-invariant we have 
\begin{equation}
  \mu_{n_k}\big([g^{-1}\cdot f_k^{-1}(A)]^{2\hspace*{0.5mm}\xi_k}\big)=\mu_{n_k}\big(g\cdot [g^{-1}\cdot f_k^{-1}(A)]^{2\hspace*{0.5mm}\xi_k}\big)
\end{equation}
Using the fact that $f_k$ is a $\xi_k$-isometry, we obtain that 
\begin{equation}
 [f_k^{-1}(A)]^{\delta_k}\subseteq f_k^{-1}(A^{\delta'_k}),
\end{equation}
where $\delta'_k:=\xi_k+\delta_k$ for all $k\in\mathbb{N}$. 
Now, using the above set-containments and equations (4.3), (4.4), (4.5), (4.6) and (4.7),  we obtain that 
$$\mu_{n_k}\big(f_k^{-1}(g^{-1}\cdot A)\big)\leq\mu_{n_k}\big(f_k^{-1}(A^{\delta'_k})\big).$$
Finally, using equation (4.2), we get 
$$(g\cdot \mu)(A)=\lim_{k\rightarrow\infty}\mu_{n_k}\big(f_k^{-1}(g^{-1}\cdot A)\big)\leq 
\lim_{k\rightarrow\infty}\mu_{n_k}\big(f_k^{-1}(A^{\delta'_k})\big)=\mu(A).$$
Applying the same argument for $g^{-1}$ and $A$, we obtain our proposition, i.e., $(g\cdot\mu)(A)=\mu(A)$ for $g\in G$ and for all Borel subsets $A$ in $X$.
\hfill\(\Box\)

\subsection*{Proof of the main theorem \ref{mainthm}:}\label{subsection4.3}

\hspace*{\fill} 

\vspace{2mm}

\textbf{Part I.} We first prove that if the box spaces are coarsely equivalent, then the groups are UME. We prove it in the following steps. 

\bigskip

\textit{Step 1:} Define $Z=H X (\subseteq W)$. Since the $H$-action on $W$ is proper and $X$ is compact, $Z$ is a closed subset of $W$. 
Clearly, $Z$ is a $G$ and $H$ invariant subset of $W$.
Therefore, we can induce the $G$ and $H$ actions from 
$W$ to $Z$. We also induce the subspace topology from $W$ to $Z$. $Y$ is a compact-open subset of $W$ and $Y\cap Z=X$. So, we obtain that 
$X_H:=X$ is a compact-open fundamental domain of $Z$ under the induced topology. 

\bigskip

\textit{Step 2:} We give the $G$-invariant probability measure $\mu$ on $X$ as constructed in Subsection 4.2 (Proposition \ref{measureconv}) and 
extend this measure on $Z$ by translating $\mu$ by 
the action of $H$. We denote the extended measure by the same symbol $\mu$. 
Therefore, we obtain a $G$ and $H$-invariant measure $\mu$ on $Z$. 

\bigskip

\textit{Step 3:} Now, we will construct a compact-open fundamental domain $X_G$ of $G$. We will follow the construction given in \cite{Sh} 
(Theorem 2.1.2, p 131). We define $E_{h}:=\{\psi\in Z:\psi(1_G)=h\}$ and  
$K_h:=G  E_{h}=\{\psi\in Z:\psi\hspace*{1mm}\mbox{takes the value}\hspace*{1mm} h\}$. Now, we enumerate the elements of $H$ by  
$h_0=1_H, h_1, h_2,\ldots$ and define 
$$X_G:= E_{1_H}\cup_{i=1}^{\infty}\big( E_{h_i}\cap K_{h_{i-1}}^c\cap\ldots\cap K_{1_H}^c\big).$$
Since $E_{1_H}=X_H$, we have $X_H\subset X_G$. We refer the reader \cite{Sh} (Theorem 2.1.2, p 131) for the proof of the fact that $X_G$ is 
a compact-open fundamenatal domain of $G$ in $Z$.

\bigskip

\textit{Step 4:} Now, it remains to show that $g$-translate of $X_H$ can be covered by finitely many $H$-translates of $X_H$ for $g\in G$ and 
$h$-translate of $X_G$ can be covered by finitely many $G$-translates of $X_G$ for $h\in H$. This easily follows from the fact that $X_G$ and $X_H$ are 
compact-open in $Z$.

By Lemma A.1 in \cite{BFS}, the composition of two UME's is a UME. Since $H\times M$ and $H$ are commensurable, 
therefore there is a UME between $G$ and $H$. 

\bigskip

\textbf{Part II.} In this part, we prove that if there is a coarse embedding of one box space into another box space, then there is a UME-embedding
of the first group into the second one. Since the proof of this part will be almost same as Part I, we will give a brief sketch of it. We will
be using same notations from the previous part. Here, we consider $Y$ as 
$\{f:G\rightarrow H|f \hspace*{1mm}\hspace*{1mm}\mbox{is a}\hspace*{1mm}\hspace*{1mm} (\rho_{+},\rho_{-})-\mbox{coarse embedding} \hspace*{1mm}\hspace*{1mm}
\mbox{and}\hspace*{1mm}\hspace*{1mm} f(1_G)=1_H\}$ and $X_n$ as $\{f:G/G_n\rightarrow H/H_n| f\hspace*{1mm}\hspace*{1mm} \mbox{is a}\hspace*{1mm}
\hspace*{1mm}$ $(\rho_{+},\rho_{-})-\mbox{coarse embedding and} \hspace*{1mm} f(1_{G/G_n})=1_{H/H_n}\}$. 
We can use a similar result like Proposition \ref{ghconv} to show that there exists a subsequence of $\{X_n\}_{n\in\mathbb{N}}$ which converges to a 
$G$-invariant compact subspace $X$ of $Y$ in Gromov-Hausdorff topology. Similarly, we can give a $G$-invariant measure $\mu$ on $X$. We define the coupling 
space as $Z=H X$ with $X_H:=X$ as the fundamental domain for $H$. Since $X_H$ is compact-open in $Z$, each $g$-translate of $X_H$
can be covered by finitely many $H$-translates of $X_H$ for all $g\in G$. Again using Lemma A.1 in \cite{BFS}, we conclude
that there is a UME-embedding from $G$ into $H$. \hfill\(\Box\)

\begin{rem}\label{fingp}
With the assumptions of our theorem, we obtain from the proof that there exists a UME between $G$ and $H\times M$ (resp. UME-embedding from $G$ into $H\times M$) 
with $X_{H\times M}\subseteq X_G$  for some finite group $M$. 
 
\end{rem}

\begin{rem}\label{margp}
The theorem \ref{mainthm} can be generalized in the setting of amenable `marked groups' in the following sense: 
Let $\{G_n\}_{n\in\mathbb{N}}$ and $\{H_n\}_{n\in\mathbb{N}}$ be two sequences of finitely generated amenable marked groups which converge to 
the finitely generated groups $G$ and $H$ respectively. Suppose there exists $(\rho_{+},\rho_{-}, c)$-coarse equivalence from $G_n$ to $H_n$ for all $n$,
where $(\rho_{+},\rho_{-}, c)$ is independent of $n$. Then, $G$ and $H$ are uniformly measured equivalent (see \cite{Da}). 

\end{rem}

\section{Applications}

The coarse embedding has been mainly studied for embedding locally finite countable metric spaces with bounded geometry inside a Hilbert space 
in the context of Baum-Connes conjecture. It arises a natural question of studying the coarse embedding inside the category of countable locally 
finite metric spaces with bounded geometry. The coarse-equivalence or quasi-isometry has been extensibly studied in geometric group theory 
among the Cayley graphs of finitely generated groups. But, we do not know much about the coarse embedding of a Cayley graph inside another Cayley graph. 
Shalom has discussed some of the results in this direction in \cite{Sh}. 

In this section, we discuss some questions of coarse equivalence and coarse embedding among box spaces. A. Naor and M. Mendel \cite{MN} first construct 
two expander sequences so that there exists no coarse embedding of one expander sequence into the other one. In their examples, one expander sequence 
has unbounded girth and another one has many short cycles. Later, D. Hume has given existence of a continuum of expander sequences with unbounded girth 
\cite{Hu}. Khukhro-Valette obtains \cite{KV} examples of another such family with bounded girth: one family is a box space of $SL_n(\mathbb{Z})$ ($n\geq 3$), a group with Property (T),  and another family is a box space of 
$SL_2(\mathbb{Z}[\sqrt{p}])$ ($p$ is a prime), a group with Haagerup property and  Property $\tau$. Using part (ii) of Corollary \ref{affinecohdim}
we prove that there is no coarse embedding of the box spaces of $SL_n(\mathbb{Z})$ into the box spaces of $SL_m(\mathbb{Z})$, where $n>m$ and $n,m\geq 3$.
Some considerations of non coarsely equivalent box spaces can be found in \cite{AG}.

\bigskip

\textbf{Proof of Corollary \ref{affinecohdim}:}

\bigskip

\textit{Proof of (i)}:  Suppose there is a coarse embedding of a box space of $G$ into a box space of $H$. Then, by Theorem \ref{mainthm},
there exists a UME-embedding of $G$ into $H$ with a compact fundamental domain $X_H$ of $G$. Suppose, $H$ has a metrically proper affine action $\alpha$ on
some $L^p(Y)$, $1\leq p\leq\infty$. We define the induced representation on $L^p(X_H; L^p(Y))\equiv L^p( X_H\times Y)$ with the induced cocycle as  
$t'(g)(x):= t(c(g,x))$ where $t$ is the 1-cocyle associated to the affine action $\alpha$ and $c:G\times X_H\rightarrow H$ is the cocycle associated 
with the $(G,H)$-coupling space $(\Omega,\mu)$. We have 
$$\|t'(g)\|=\big(\int_{X_G} ||t(c(g,x))||^p d\mu(x)\big)^{1/p}<\infty$$   
Now, we also have $\rho_{-}(|g|)\leq |c(g,x)| \leq \rho_{+}(|g|)$ for all $x\in X_H$. Therefore, $\|t'(g)\|\geq \big(\mu(X_H)\big)^{1/p} 
\|t\big(\rho_{-}(g)\big)\|$.
Since $\rho_{-}$ is a proper function and $t$ is a proper 1-cocyle,  $t'$ is a proper 1-cocyle too. Hence, we have our corollary. \hfill\(\Box\)

\bigskip

\textit{Proof of (ii):} By Theorem \ref{mainthm} and Remark \ref{fingp}, there exist a finite group $M$ and a UME-embedding of $G$ into $H\times M$ with a compact fundamental domain 
$X_{H\times M}$ of $H\times M$, which is inside the fundamental domain $X_G$ of $G$. Since $cd_R(H\times M)=cd_R H$, without loss of generality, we can 
replace $H\times M$ by $H$. Now, we can use the same argument as given in the proof of Theorem 1.5 in \cite{Sh} to conclude that $cd_R G\leq cd_R H$. 
\hfill\(\Box\)

\bigskip

\subsection{Distinguishing box spaces up to coarse embedding}

We are now ready to discuss the questions of coarse embedding among the box spaces of the following classes of groups. Some of the already known examples 
can be deduced from Corollary \ref{affinecohdim}.

\bigskip 

\textbf{I. $SL_n(\mathbb{Z})$ and $SL_m(\mathbb{Z})$, where $n>m$ and $n,m\geq 3$:}
 
 Using part (ii) of Corollary \ref{affinecohdim}, we get that if there exists a coarse embedding of a box space of $SL_n(\mathbb{Z})$ inside $SL_m(\mathbb{Z})$, then
 $cd_R SL_n(\mathbb{Z}) \leq cd_R SL_m(\mathbb{Z})$. But, by a result of Borel-Serre \cite{BS}, we get explicit values of cohomological dimension of 
 these groups with coeffcients in the ring $\mathbb{Q}$ : $cd_\mathbb{Q} SL_n(\mathbb{Z})=dim N$, where $N$ is set the upper triangular 
 unipotent matrices
 which appears in the Iwasawa decomposition of $SL_n(\mathbb{R})$. Therefore, there exists no coarse embedding of the expanders obtained from 
 $SL_n(\mathbb{Z})$ into the expanders obtained from $SL_m(\mathbb{Z})$, where $n>m$. 
 
 This also can be proved in the following way. By a remark, given in Subsection 6.1 of \cite{Sh}, if there exists a coarse embedding of 
 $SL_n(\mathbb{Z})$ into $SL_m(\mathbb{Z})$, $cd_{\mathbb{Q}} SL_n(\mathbb{Z})\leq cd_{\mathbb{Q}} SL_m(\mathbb{Z})$. We recall the argument briefly: 
 The intersection of $N$ with $SL_n(\mathbb{Z})$, denoted by $N(\mathbb{Z})$, is coarsely embedded in $SL_n(\mathbb{Z})$.
 Now, by Theorem 1.5 in \cite{Sh}, if $\Lambda$ is amenable and if $\Lambda$ coarsely embeds in $\Gamma$, then $cd_{\mathbb{Q}}\Lambda\leq cd_{\mathbb{Q}}\Gamma$.
 However, $cd_{\mathbb{Q}} N(\mathbb{Z})=cd_{\mathbb{Q}} SL_n(\mathbb{Z})=dim N$. Now, by using Proposition \ref{coarseequi},
 we have the above result.

\bigskip
 
 \textbf{II. Groups with Property T and  groups with Haagerup Property:} The groups with Property (T) have the property that 
 any affine isometric action of these groups on a Hilbert space has a fixed point. But, the groups with Haagerup property have metrically proper 
 affine isometric action on a Hilbert space. Therefore, the part (i) of Corollary \ref{affinecohdim}
  easily distinguishes the box spaces of these two groups up to coarse embedding.  This also can be proved by using the following result :
 a finitely generated, residually finite group has Haagerup property if and only if one (or equivalently, all) of its box spaces admits a fibred coarse
 embedding into a Hilbert space \cite{CWW}. This argument has been used by Khukhro-Valette in \cite{KV} to show that there exists no coarse 
 embedding of the box spaces of $SL_n(\mathbb{Z})$, where $n\geq 3$, into the box spaces of $SL_2(\mathbb{Z}[\sqrt{p}])$, where $p$ is a prime.

\bigskip

\textbf{III. Lattices of $SL_n(\mathbb{R})$ ($n\geq 3$) and a finite product of hyperbolic groups:} 
 
By \cite{BFGM}, the lattices in $SL_n(\mathbb{R})$ have Property $F_{L^p}$, where $1<p<\infty$, i.e., any affine isometric action of these groups 
on an $L^p$-space  ($1<p<\infty$) has a fixed point. On the other hand, by \cite{Yu2}, hyperbolic groups admit a metrically proper affine isometric action on an
$l^p$-space for some $2\leq p <\infty$. Therefore, by part (ii), there is no coarse embedding of the box spaces of lattices in $SL_n(\mathbb{R})$
inside the box spaces of a finite product of hyperbolic groups. This also can be argued by using a result about `fibred coarse embedding inside an $L^p$-space', 
where $1\leq p\leq\infty$:  a finitely generated, residually finite group has a proper affine action on an $L^p$-space if and only if one 
(or equivalently, all) of its box spaces admits a fibred coarse embedding into an $L^p$-space \cite{Ar}. In particular, since the lattices in 
$Sp(n,1)$ ($n\geq 2$) are  
hyperbolic groups  with Property (T), using our result we obtain that 
there is no coarse embedding of the expander sequences coming from $SL_n(\mathbb{Z})$ ($n\geq 3$) into the expanders coming from the lattices of 
$Sp(n,1)$  ($n\geq 2$).

\bigskip

\subsection{Distinguishing box spaces up to coarse equivalence:}
In this subsection, we show a countable class of coarsely equivalent groups such that the box spaces of two such groups 
are not coarsely equivalent.
The examples of residually finite groups which are quasi-isometric but not measured equivalent will serve our purpose. 
The following class of examples has been suggested by D. Gaboriau: Let $\Gamma_{p,q,r}:=(F_p\times F_q)\ast F_r$ , where $F_p$, $F_q$ and $F_r$ are 
free groups with $p$, $q$ and $r$ generators, respectively, and $p, q, r\geq 2$.

\begin{cor}\label{notqiboxsp}
 There exists a countable class of residually finite groups from the collection 
 $$\{\Gamma_{p,q,r} | (p,q,r)\in\mathbb{N}\times\mathbb{N}\times\mathbb{N} \}$$
 where any two groups are coarsely equivalent but any of their box spaces are not coarsely equivalent. 
\end{cor}
\begin{proof}

Since all free groups with at least two generators are commensurable, the groups $\Gamma_{p,q,r}$ for different $(p,q,r)$ are quasi-isometric. 
Using Properties 1.5 and Example 1.6 of \cite{Ga} (p. 12, 13), we compute the first and second $l^2$-betti numbers of this group:
$\beta_1^{(2)}[(F_p\times F_q)\ast F_r]=r$ and  $\beta_2^{(2)}[(F_p\times F_q)\ast F_r]=(p-1)(q-1)$. By Corollary \ref{l2betti}, the betti numbers of 
$G$ and $H$ must be proportional. Therefore, we can obtain a countable class of residually finite groups from 
$\{\Gamma_{p,q,r} | (p,q,r)\in\mathbb{N}\times\mathbb{N}\times\mathbb{N} \}$ which are not mutually measured equivalent. Now, 
using Theorem \ref{mainthm}, we obtain that any two box spaces of two diffrent groups from the abovementioned class are not coarsely equivalent.

\end{proof}

\begin{rem}
The above corollary shows that Theorem \ref{mainthm} is stronger than Proposition \ref{coarseequi}, i.e., there are examples of groups whose box spaces can
not be distinguished up to coarse-equivalence by \ref{coarseequi} but can be distinguished by \ref{mainthm}. 
\end{rem}

\section{Some questions}

 The author does not know \textit{`whether there exist two residually finite groups $\Gamma$ and $\Lambda$ such that there exists a coarse embedding 
of $\Gamma$ into $\Lambda$ but there exists no UME-embedding of $\Gamma$ into $\Lambda$'}. If such examples exist, this will assure that Theorem \ref{mainthm}
is stronger than Proposition \ref{coarseequi} in terms of distinguishing box spaces up to coarse embedding. 
We suspect that this might be a possible example:  We know from above that there is no coarse embedding of the box spaces of a lattice inside 
a simple Lie group of higher rank, for example $SL_n(\mathbb{Z})$ ($n\geq 3$), into the box spaces of a finite product of hyperbolic groups. But, 
it is not known to the author whether there exists a coarse embedding of $SL_n(\mathbb{Z})$ ($n\geq 3$) into a finite product of hyperbolic groups or into a finite
product of $3$-regular trees. A similar question was asked by Cornulier in \cite{Co} (Question 1.12, p. 5). 

\bigskip

 The author would also like to know whether the converse of the main theorem \ref{mainthm} is true, i.e., 
\textit{is it true that if two residually finite groups are uniform measured equivalent, then there exist box spaces of each of them which are 
coarsely equivalent?}

\end{document}